\documentclass[12pt]{article}
\usepackage{amsmath}
\usepackage{amsfonts}
\usepackage{amsthm}
\usepackage{amssymb}
\usepackage{color}
\usepackage{enumerate}
\usepackage[mathscr]{euscript}

\newtheorem{theorem}{Theorem}[section]
\newtheorem{lemma}[theorem]{Lemma}

\newtheorem{example}[theorem]{Example}

\theoremstyle{definition}

\theoremstyle{remark}
\newtheorem{remark}[theorem]{Remark}

\newcommand{\sw}{{\textsf{w}}\hskip0.01cm}
\newcommand{\sx}{{\textsf{x}}\hskip0.01cm}
\newcommand{\sy}{{\textsf{y}}\hskip0.01cm}

\newcommand{\tp}{{\rm TP}\hskip0.02cm}
\newcommand{\tc}{{\rm TC}\hskip0.02cm}

\def\mh{\mathcal{H}}

\def\fs{\mathfrak{s}}
\def\fr{\mathfrak{r}}

\def\mju{\mathcal{U}}

\def\f2{\mathbb{F}_2}

\def\lip{\hskip0.02cm{\rm Lip}\hskip0.01cm}
\def\supp{\hskip0.02cm{\rm supp}\hskip0.01cm}

\newcommand{\lin}{{\rm lin}\hskip0.02cm}

\newcommand{\1}{\mathbf{1}}

\newcommand\remove[1]{}

\begin{document}

\title{\LARGE Complementability of isometric copies of $\ell_1$ in transportation cost spaces}

\author{Sofiya Ostrovska and Mikhail~I.~Ostrovskii}

\date{\today}
\maketitle

\begin{abstract} This work aims to establish new results pertaining to the structure of transportation cost spaces.
Due to the fact that those spaces were studied and applied in
various contexts, they have also become known under different
names such as Arens-Eells spaces, Lipschitz-free spaces, and
Wasserstein spaces. The main outcome of this paper states that if
a metric space $X$ is such that the transportation cost space on
$X$ contains an isometric copy of $\ell_1$, then it contains a
$1$-complemented isometric copy of $\ell_1$.
\end{abstract}

{\small \noindent{\bf Keywords.} Arens-Eells space, Banach space,
earth mover distance, Kantorovich-Rubin\-stein distance,
Lipschitz-free space, transportation cost, Wasserstein distance}

{\small \noindent{\bf 2020 Mathematics Subject Classification.}
Primary: 46B04; Secondary: 46B20, 46B85, 91B32}



\bigskip

\hfill{\sf In memory of all the people who have sacrificed}

\hfill{\sf their lives while fighting for Ukraine since 2014}

\section{Introduction}

In this paper we continue the study of Banach-space-theoretical
properties of transportation cost spaces. The study of
transportation cost spaces was launched by Kantorovich
\cite{Kan42}, see also \cite{KG49}. As time passed, these spaces
have proven to possess the high degree of importance within a
variety of directions. This, in turn, have led to the diversity of
names used for the spaces, the most popular names are mentioned in
the Abstract. We stick to the term  {\it transportation cost
space} since, in our opinion,  it immediately clarifies the circle
of discussed problems and is consistent with the history of the
subject. A detailed survey on the development of those notions
along with relevant historical comments is presented in
\cite[Section 1.6]{OO19}.

Before we begin, let us recall some necessary definitions and
facts. Let $(X,d)$ be a metric space.  If $f:X\rightarrow
\mathbb{R}$ is a function possessing a finite support and
satisfying the condition $\sum_{v\in \supp f}f(v)=0$, then, in a
natural way,  it can be  viewed  as a \emph{transportation
problem} (on $X$) of certain product from sites where it is
available  ($f(v)>0$) to those where it is demanded ($f(v)<0$).

Every transportation problem $f$ admits a presentation of the
form:
\begin{equation}\label{E:TranspPlan} f=a_1(\1_{x_1}-\1_{y_1})+a_2(\1_{x_2}-\1_{y_2})+\dots+
a_n(\1_{x_n}-\1_{y_n}),\end{equation} where $a_i\ge 0$,
$x_i,y_i\in X$, and $\1_u(x), u\in X$ stands for the {\it
indicator function} of $u$. Since equality \eqref{E:TranspPlan}
can be regarded as a plan of carrying  $a_i$ units of the product
from $x_i$ to $y_i$, every  representation of this form is said to
be a \emph{transportation plan} for $f$. In this interpretation,
the sum $\sum_{i=1}^n a_id(x_i,y_i)$ defines the
 \emph{cost} of that plan.

In the sequel, $\tp(X)$ denotes the real vector space of all
transportation problems (on $X$). We endow $\tp(X)$ with the
\emph{transportation cost norm} (or \emph{transportation cost},
for short). Namely, for $f\in\tp(X)$, the norm $\|f\|_{\tc}$ is
defined as the infimum of costs taken over all transportation
plans given by \eqref{E:TranspPlan}.

For infinite $X$, the space $\tp(X)$  with $\|\cdot\|_{\tc}$ may
not be complete, its completion is called the \emph{transportation
cost space}  and denoted by $\tc(X)$. When  $X$ is finite, the
above spaces $\tc(X)$ and $\tp(X)$ are identical  as sets. The
notation $\tc(X)$ is employed to highlight the
 normed vector space structure when we need it.
\medskip

It can be effortlessly derived from the triangle inequality that,
whenever $f\in\tc(X)$ has a finite support, the infimum of costs
of transportation plans   is attained. Moreover, it is easy to
notice that   this occurs for a  plan with $\{x_i\}=\{v:~f(v)>0\}$
and $\{y_i\}=\{v:~f(v)<0\}$. Notice, that a transportation plan
that provides the infimum need not be unique. Any such a
transportation plan for $f$, whose cost equals $\|f\|_\tc$ is
called  an \emph{optimal transportation plan} for $f\in\tp(X)$
because this plan has the minimal possible cost. See \cite{OO22}
for a more detailed introduction.

Prior to presenting our results, it seems appropriate to outline
the motivation for studying  transportation cost spaces:

(1) The dual space of $\tc(X)$ is the space of Lipschitz functions
on $X$ vanishing at a specified point and equipped with its
natural norm. This makes the space $\tc(X)$ an object of Classical
Analysis, especially in cases where $X$ is a classical metric
space like $\mathbb{R}^n$.

(2) A metric space $X$ admits a canonical isometric embedding into
$\tc(X)$ (Arens--Eells observation \cite{AE56}). This fact makes
$\tc(X)$ a natural object of study in the theory of Metric
Embeddings, see \cite[Chapter 10]{Ost13}.

(3) The norm in this space can be interpreted as a transportation
cost.

(4) The transportation cost space $\tc(X)$ can be regarded as a
kind of a {\it linearization} of $X$, and can be used to
generalize Banach-space-theoretical notions to the case of metric
spaces. This approach was suggested by Bill Johnson; later his
idea was described in \cite[p.~223]{Bou86}, where its limitations
were discovered. See also the discussion in \cite{Nao18}.

(5) Transportation cost spaces were applied  to solve some
important problems of the  Banach space theory, both of linear and
non-linear theories. The respective program was initiated by
Godefroy-Kalton in \cite{GK03} and  significantly developed by
Kalton in  \cite{Kal04}.

In this paper,   the study of Banach space geometry of $\tc(X)$ is
continued. More specifically, we focus at  studying the relations
between its structure and the structure of the space $\ell_1$.
Available related results can be found in
 \cite{AACD20,APP21,CDW16,CJ17,DKP16,DKO20,God10,KMO20,OO19,OO20}.

As the most closely related predecessors of this research, the
following   results have to be cited:

\begin{theorem}[\cite{KMO20}]\label{T:2npts}If a metric space $X$ contains $2n$ elements, then $\tc(X)$ contains a
$1$-complemented subspace isometric to $\ell_1^n$.
\end{theorem}

To formulate the next theorem, we introduce, for any finite set
$\{v_i\}_{i=1}^m$ in $(X,d)$, the complete weighted graph
$K(\{v_i\}_{i=1}^m)$ with vertex set $\{v_i\}_{i=1}^m$ and with
the weight of an edge $uv$ equal to $d(u,v)$.

\begin{theorem}[\cite{OO20}]\label{T:Contell_1} The space $\tc(X)$ contains $\ell_1$ isometrically
if and only if there exists a sequence of pairs
$\{x_i,y_i\}_{i=1}^\infty$ in $X$, with all elements distinct,
such that each set $\{x_iy_i\}_{i=1}^n$ of edges is a minimum
weight perfect matching in the $K(\{x_i,y_i\}_{i=1}^n)$.
\end{theorem}

The assertion  below is the main result of this paper.

\begin{theorem}\label{T:Main} If a metric space $X$ is such that $\tc(X)$
contains a subspace isometric to $\ell_1$, then $\tc(X)$ contains
a $1$-complemented isometric copy of $\ell_1$.
\end{theorem}

\begin{remark}\label{rem1} In general,  a linear isometric copy
of $\ell_1$ in $\tc(X)$ does not have to be complemented. This
fact is a consequence of the following result proved in
\cite[Theorem 3.1]{GK03}: There exists a metric space $X_C$ such
that $\tc(X_C)$ contains a linear isometric copy of $C[0,1]$.  To
show that $\tc(X_C)$ contains a linearly isometric copy of
$\ell_1$ which is not complemented, it suffices to combine this
result with  the two classical facts: (i) $\ell_1$ admits a linear
isometric embedding into $C[0,1]$ (Banach-Mazur, see \cite[Theorem
9, p.~185]{Ban32}), (ii) The image of this subspace is not
complemented, for example, because the dual of $C[0,1]$ is weakly
sequentially complete, but the dual of $\ell_1$ is not; see
\cite[Chapter IV]{DS58}.
\end{remark}

Remark \ref{rem1}, which establishes the existence of
non-complemented linear isometric copies of $\ell_1$ in  $\tc(X)$,
is based on important classical results. However, if we are
interested in subspaces isometric to $\ell_1$ which are only
not
$1$-complemented, such example can be constructed in a more
elementary way. We present such an example below.

\begin{example} There is a simple metric space $X_K$ such that $\tc(X_K)$
contains a linear isometric copy of $\ell_1$ which is not
$1$-complemented.
\end{example}

Recall that $K_{4,4}$ is a complete bipartite graph in which both
parts have $4$ vertices. The starting point of this example is the
fact \cite[Section 8]{DKO21} that $\tc(K_{4,4})$  contains a
linearly isometric copy of $\ell_\infty^4$. It is a well-known
observation of Gr\"unbaum \cite{Gru60}, that there exist a
subspace in $\ell_\infty^4$ which is isometric to $\ell_1^3$ and
is not $1$-complemented. We let $X_K$ be the union of the vertex
set $V(K_{4,4})$ and $\mathbb{N}$ (with their usual metrics).
Next, pick a vertex $O$ in $V(K_{4,4})$ and introduce the metric
on the union as follows: the distance between two points in
$V(K_{4,4})$ or $\mathbb{N}$ is equal to the original. The
distance between $v\in V(K_{4,4})$ and $m\in\mathbb{N}$ is equal
to $d(v,O)+m$. It is well-known  (see \cite[Section 3.1]{AFGZ21})
that for such metric one has
$\tc(X_K)=\tc(K_{4,4})\oplus_1\ell_1$. By the aforementioned
example of \cite{DKO21}, this space contains a subspace isometric
$\ell_\infty^4\oplus_1\ell_1$. Thus, by the observation of
\cite{Gru60} stated above, the space contains a linear isometric
copy of $\ell_1$ which is not $1$-complemented.

\section{Proof of Theorem \ref{T:Main}}

Our proof of Theorem \ref{T:Main} is based on Theorem
\ref{T:Contell_1}. The following result is the key lemma in the
proof of Theorem \ref{T:Main}.

\begin{lemma}\label{L:FiniteMatch} Let $(X,d)$ be a metric space containing a set
$\{x_i,y_i\}_{i=1}^n$ of pairs forming a minimum-weight matching
in $K(\{x_1,\dots,x_n,y_1,\dots,y_n\})$. Then, there exists a
surjective norm-$1$  projection $P_n:\tc(X)\to L_n$, where $L_n$
is the subspace of $\tc(X)$ spanned by
$\{\1_{x_i}-\1_{y_i}\}_{i=1}^n$.
\end{lemma}

\begin{proof} Denote the vector $\frac{\1_{x_i}-\1_{y_i}}{d(x_i,y_i)}$ by
$\fs_i$. To prove this lemma,  we are going to construct a
sequence $\{t_{i,n}\}_{i=1}^n$ of $1$-Lipschitz functions on $X$
such that $\{\fs_i,t_{i,n}\}_{i=1}^n$ is a biorthogonal set, and
also
\begin{equation}\label{E:P_n}
P_n(f):=\sum_{i=1}^nt_{i,n}(f)\fs_i
\end{equation}
is a surjective norm-$1$ projection $P_n:\tc(X)\to L_n$.

\begin{remark}\label{R:NoZeroAtO} The dual of the space $\tc(X)$ is identified as $\lip_0(X)$ - the
space of Lipschitz functions on $X$ having value $0$ at a picked
and fixed point $O$ in $X$, called the {\it base point} (see
\cite[Chapter 10]{Ost13}). Nevertheless, any Lipschitz function
$t$ on $X$ gives rise to a continuous linear functional on
$\tc(X)$, the functional is the same as the functional produced by
$t-t(O)\in\lip_0(X)$. Because of this, in the selection of
$t_{i,n}$ the condition $t_{i,n}(O)=0$ may be dropped out.
\end{remark}

At this point, we notice that, after establishing the
biorthogonality, it suffices to prove that $\|P_n(f)\|_\tc\le
\|f\|_\tc$ for every $f\in\tc(X)$ of the form $f=\1_w-\1_z$ for
$w,z\in X$. This will be shown by using the reasoning of
\cite[p.~196]{KMO20}. For the convenience of the reader, the
details are presented below. Indeed, by the definition of
$\tc(X)$, the space $\tp(X)$ is dense in $\tc(X)$, and for $g \in
\tp(X)$ the desired inequality can be derived from the case
$f=\1_w-\1_z$ as follows. For each $g\in\tp(X)$, there exists an
optimal transportation plan (see Section 1 or \cite[Proposition
3.16]{Wea18}). Hence, $g$ can be represented as a sum
$g=\sum_{i=1}^m g_i$, where all $g_i$ are of the form
$g_i=b_i(\1_{w_i}-\1_{z_i})$, $b_i\in\mathbb{R}$, and
$\|g\|_\tc=\sum_{i=1}^m\|g_i\|_\tc$. Therefore, assuming that we
proved the inequality $\|P_n(f)\|_\tc\le \|f\|_\tc$ in the case
$f=\1_w-\1_z$,  we obtain:
\[\|P_ng\|_\tc=\left\|P_n\left(\sum_{i=1}^m g_i\right)\right\|_\tc\le\sum_{i=1}^m\|P_ng_i\|_\tc\le \sum_{i=1}^m\|g_i\|_\tc=\|g\|_\tc,
\]
and, thus, $\|P_n\|\le 1$.
\medskip

To construct $\{t_{i,n}\}_{i=1}^n$, we need to restate the
assumption that $\{x_iy_i\}_{i=1}^n$ is a minimum weight matching
in $K(\{x_1,\dots,x_n,y_1,\dots,y_n\})$ in Linear Programming (LP)
terms. Originally, this approach was suggested by Edmonds
\cite{Edm65}. Below, we follow the presentation of this approach
given in \cite[Sections 7.3 and 9.2]{LP09}. First, consider the
minimum weight perfect matching problem on a complete weighted
graph $G$ with even number of vertices and weight $\sw:E(G)\to
\mathbb{R}$, $\sw\ge 0$. By \cite[Theorem 7.3.4]{LP09}, the
minimum weight perfect matching problem can be reduced to the
linear program {\bf (LP1)} described below. Within the program, an
{\it odd cut}  designates the set of edges in $G$  joining a
subset of $V(G)$ of odd cardinality with its complement,  while a
{\it trivial odd cut} designates  a set of edges joining one
vertex with its complement. If $\sx$ is a real-valued function on
$E(G)$ and $A$ is a set of edges, we define $\sx(A):=\sum_{e\in
A}\sx(e)$. The reduction means that the linear program has an
integer solution corresponding to a minimum weight perfect
matching.
\medskip

Here comes the program.

\begin{itemize}

\item {\bf (LP1)} minimize $\sw^\top \cdot \sx$  (where
$\sx:E(G)\to \mathbb{R}$)
\medskip

\item {\color{black} subject to}

\begin{enumerate}[{\bf (1)}]

\item $\sx(e)\ge 0$ for each $e\in E(G)$

\item $\sx(C)=1$ for each trivial odd cut $C$

\item\label{I:3} $\sx(C)\ge 1$ for each non-trivial odd cut $C$.

\end{enumerate}
\end{itemize}

Next, we  introduce a variable $\sy_C$ for each odd cut $C$.
\medskip

The dual program of the program {\bf (LP1)} is:

\begin{itemize}

\item {\bf (LP2)} maximize $\sum_C \sy_C$

\item subject to

\begin{itemize}

\item[{\bf (D1)}] $\sy_C\ge 0$ for each non-trivial odd cut $C$

\item[{\bf (D2)}] $\sum_{C~ {\rm containing~ }e} \sy_C\le \sw(e)$
for every $e\in E(G)$.

\end{itemize}
\end{itemize}

The Duality in Linear Programming \cite[Section 7.4]{Sch86} - see
also a summary in \cite[Chapter 7]{LP09} -  states that the optima
{\bf (LP1)} and {\bf (LP2)} are equal. Therefore, the total weight
of the minimum weight perfect matching coincides with the sum of
entries of the optimal solution of the dual problem.
\medskip

In order to proceed, it is beneficial to recall some of the
properties of optimal solutions $\{\sy_C\}$. Let $M$ be a minimum
weight perfect matching in $G$. We start with the following
observation:
\begin{equation}\label{E:ForOptDual} \sw(M)=\sum_{e\in M} \sw(e)\substack{{\bf (D2)}\\\ge} \sum_{e\in
M}~~\sum_{C~ {\rm containing~ }e} \sy_C=\sum_{C}|M\cap
C|\sy_C\substack{{\bf (3)}\\\ge}\sum_C \sy_C,\end{equation} where
we use the fact that $|M\cap C|\ge 1$ for each perfect matching
$M$ and each odd cut $C$. See \cite[p.~371]{LP09}.

If $\sy_C$ is an optimal dual solution, then the leftmost and the
rightmost sides in \eqref{E:ForOptDual} coincide, implying
\begin{equation}\label{E:WeighAtt} \sw(e)=\sum_{C~ {\rm
containing~ }e}\sy_C\end{equation} for each $e\in M$ and
\begin{equation}\label{E:Int1} |M\cap C|=1 \hbox{ for each non-trivial odd cut }C
\hbox{ satisfying }\sy_C>0.\end{equation}

Analysis in \cite[p.~372--374]{LP09} shows that we may assume that
there exists a family $\mh$ of subsets of $V(G)$ which satisfies
the conditions:

\begin{itemize}

\item[(P-1)] $\mh$ is \emph{nested} in the sense that for any
$D,T\in \mh$ either $D\subseteq T$ or $T\subseteq D$ or $D\cap
T=\emptyset$.

\item[(P-2)] $\mh$ contains all singletons of $V(G)$.

\item[(P-3)] if $C$ is a non-trivial odd cut, then $y_C>0$ if and
only if $C=\partial D$ for some $D\in\mh$, where $\partial D$ is
the set of edges connecting $D$ and $V(G)\backslash D$.

\end{itemize}

Furthermore, \cite[Theorem 5.20]{CCPS98} and \cite[Lemma
14.11]{KMO20} established that if the weight function $\sw$
satisfies $\sw(uv)=d(u,v)$ for some metric $d$ on $V(G)$ and all
$u,v\in V(G)$, then there is an optimal dual solution satisfying
also $\sy_C\ge 0$ for all odd cuts, including trivial ones.
\medskip

These results  will be applied to the weighted graph
$G=K(\{x_1,\dots,x_n,y_1,\dots,y_n\})$ and the matching
$\{x_iy_i\}_{i=1}^n$. We denote the matching $\{x_iy_i\}_{i=1}^n$
by $M_n$, the graph $K(\{x_1,\dots,x_n,y_1,\dots,y_n\})$ by
$K(M_n)$, and its vertex set by $V_n$.
\medskip

Keeping the notation $\mh$ for the obtained nested family of
subsets of $V_n$ satisfying (P-1)--(P-3), we may and shall assume
that all elements in $\mh$ have cardinalities at most $n$, due to
the fact that each edge boundary of a set is a boundary of a set
having such cardinality, and that $\mh$ contains at most one set
of cardinality $n$ (see condition (P-1)). With these assumptions
the correspondence between the edge boundaries of sets in $\mh$
and the cuts $C$ which are either trivial or satisfy $\sy_C>0$ is
bijective.
\medskip

With this in mind, it is only a slight abuse of notation to denote
the weight of $\partial D$ by $\sy_D$, in particular,
$\sy_{\{v\}}$ for a vertex $v$ denotes the weight of the trivial
cut separating vertex $v$ from the rest of $V_n$ (in $K(M_n)$).
\medskip

Our next goal is to construct $1$-Lipschitz functions
$\{t_{i,n}\}_{i=1}^n$ satisfying the conditions
$t_{i,n}(y_i)-t_{i,n}(x_i)=d(x_i,y_i)$ and
$t_{i,n}(y_j)-t_{i,n}(x_j)=0$, $i,j\in\{1,\dots, n\}$, $j\ne i$.
Some features of this construction will be used to prove the
inequality $\|P_n(f)\|_\tc\le\|f\|_\tc$.\medskip

Using the notation $B_X(v,\fr)=\{x\in X:~ d(x,v)\le \fr\}$ for
$\fr>0$, we define, for each $F\in\mh$,  the set
\[U_F=\bigcup_{v\in F}B_X\left(v,\sum_{D\subseteq F,~v\in D\in \mh}\sy_D\right).\]
Note that for a $1$-element set $F=\{v\}$, $v\in V_n$, one has
$U_F=B_X(v,\sy_F)$.
\medskip

As a next step, we introduce three collections of $1$-Lipschitz
functions: $r_{\lambda,\theta,H}:X\to\mathbb{R}$ and
$s_{\lambda,\theta,H}:X\to\mathbb{R}$ parameterized by
$\lambda\in\mathbb{R}$, $\theta=\pm 1$, and $H\in\mh$, and the
collection $t_{D,F}$ parameterized by $D,F\in\mh$. Here is the
definition of $r_{\lambda,\theta,H}$:

\begin{equation}\label{E:Funct_r}
r_{\lambda,\theta,H}(x)=\lambda+\theta\min_{v\in
H}\{\max\{(d(x,v)-\sum_{v\in D\subsetneq H,~ D\in\mh}\sy_D),0\}\}.
\end{equation}

In the case where $H=\{v\}$, we understand this formula as

\[
r_{\lambda,\theta,\{v\}}(x)=\lambda+\theta d(x,v).\]

Note that $r_{\lambda,\theta,H}$ is equal to $\lambda$ on
$\bigcup_{D\subsetneq H,~D\in\mh}U_D$. The function
$r_{\lambda,\theta,H}$ is $1$-Lipschitz because $d(x,v)$ is
$1$-Lipschitz and all operations which we apply to it, namely,
maximum, minimum, multiplication with $\pm 1$ and addition of a
constant preserve this property.
\medskip

Now we define $1$-Lipschitz functions
$s_{\lambda,\theta,H}:X\to\mathbb{R}$ parameterized by
$\lambda\in\mathbb{R}$, $\theta=\pm 1$, and $H\in\mh$:
\begin{equation}\label{E:Funct_s}
s_{\lambda,\theta,H}(x)=\lambda+\theta(\min\{\min_{v\in
H}\{\max\{(d(x,v)-\sum_{v\in D\subsetneq H,~
D\in\mh}\sy_D),0\}\},\sy_H\}).
\end{equation}

In the case where $H=\{v\}$, the definition means the following:
\[
s_{\lambda,\theta,\{v\}}(x)=\lambda+\theta\min\{d(x,v),\sy_{\{v\}}\}.\]

Note that $s_{\lambda,\theta,H}$ is equal to $\lambda$ on
$\bigcup_{D\subsetneq H,~D\in\mh}U_D$ and to $\lambda+\theta
\sy_H$ outside $U_H$. The function $s_{\lambda,\theta,H}$ is
$1$-Lipschitz for the same reason as $r_{\lambda,\theta,H}$.

Now we start constructing functions $t_{D,F}$ for different
subsets $D,F \in\mh$. As an initial point, let $D=D_1=\{x_i\},
F=F_1=\{y_i\}$, and $x_iy_i$ be an edge in the matching $M_n$. In
this case, we denote $t_{D,F}$ also $t_{i,n}$ because these
$1$-Lipschitz functions will be the desired biorthogonal
functionals for $\{\fs_i\}_{i=1}^n$.

Let $D_1=\{x_i\}\subsetneq D_2\subsetneq
D_3\subsetneq\dots\subsetneq D_{\tau_i}$ be elements of $\mh$,
where $D_{\tau_i}$ is the largest set in $\mh$ containing $x_i$
but not containing $y_i$. Assume also that this increasing
sequence is maximal in the sense that there is no $J\in\mh$
satisfying $D_k\subsetneq J\subsetneq D_{k+1}$.

Similarly, let $F_1=\{y_i\}\subsetneq F_2\subsetneq
F_3\subsetneq\dots\subsetneq F_{\sigma_i}$ be a maximal increasing
sequence of sets in $\mh$ with $x_i\notin F_{\sigma_i}$.

  We define:
\begin{equation}\label{E:In_r}
t_{i,n}(x)=\begin{cases} l_{i,n}(x) &\hbox{ if } l_{i,n}(x)<
\sy_{D_1}+\dots+\sy_{D_{\tau_i}}\\
h_{i,n}(x) &\hbox{ if } h_{i,n}(x)>
\sy_{D_1}+\dots+\sy_{D_{\tau_i}}\\
\sy_{D_1}+\dots+\sy_{D_{\tau_i}} &\hbox{ otherwise } \end{cases}
\end{equation}
where
\begin{equation}\begin{split}
l_{i,n}(x)=\min\{r_{0,1,D_1}(x),& r_{\sy_{D_1},1,D_2}(x),
r_{\sy_{D_1}+\sy_{D_2},1,D_3}(x), \dots,\\&\qquad
r_{\sy_{D_1}+\dots+\sy_{D_{\tau_i-1}},1,D_{\tau_i}}(x),
\}\end{split}
\end{equation}
and
\begin{equation}\begin{split}
h_{i,n}(x)=&\max\{r_{\sy_{D_1}+\dots+\sy_{D_{\tau_i}}+\sy_{F_{\sigma_i}},-1,F_{\sigma_i}}(x),\dots,
\\&~ r_{\sy_{D_1}+\dots+\sy_{D_{\tau_i}}+\sy_{F_{\sigma_i}}+\dots+\sy_{F_2},-1,F_2}(x),
r_{\sy_{D_1}+\dots+\sy_{D_{\tau_i}}+\sy_{F_{\sigma_i}}+\dots+\sy_{F_1},-1,F_1}(x)
\}.\end{split}
\end{equation}

It is not obvious that $t_{i,n}$ is well-defined, but it follows
from the presented below proof that $t_{i,n}$ is $1$-Lipschitz.
\medskip

The functions $l_{i,n}$ and $h_{i,n}$, are $1$-Lipschitz because
they have been obtained from $1$-Lipschitz functions using the
maximum and minimum operations. As for $t_{i,n}$, it suffices to
verify the $1$-Lipschitz condition for $x$ and $y$ satisfying
$l_{i,n}(x)< \sy_{D_1}+\dots+\sy_{D_{\tau_i}}$ and $h_{i,n}(y)>
\sy_{D_1}+\dots+\sy_{D_{\tau_i}}$.

Since all minima in the definitions above are over finite sets, we
may, without loss of generality, assume that that there exist
$k\in\{0,\dots,\tau_i\}$ and  $u\in D_{k+1}$ such that

\begin{equation}\label{E:lin-choice}\begin{split} l_{i,n}(x)&=r_{\sy_{D_1}+\dots+\sy_{D_k},1,D_{k+1}}(x)\\&=\sy_{D_1}+\dots+\sy_{D_k}+\max\{(d(x,u)-\sum_{u\in
D\subsetneq D_{k+1},~ D\in\mh}\sy_D),0\}.\end{split}\end{equation}

Similarly, without loss of generality, we may assume that there
exist $m\in\{0,\dots,\sigma_i\}$ and $w\in F_{m+1}$ such that
\begin{equation}\label{E:hin-choice}\begin{split}
h_{i,n}(y)&=r_{\sy_{D_1}+\dots+\sy_{D_{\tau_i}}+\sy_{F_{\sigma_i}}+\dots+\sy_{F_{m+1}},-1,F_{m+1}}(x)
\\&=\sy_{D_1}+\dots+\sy_{D_{\tau_i}}+\sy_{F_{\sigma_i}}+\dots+\sy_{F_{m+1}}
\\&-\max\{(d(y,w)-\sum_{w\in F\subsetneq F_{m+1},~
F\in\mh}\sy_F),0\}.\end{split}\end{equation}

If both maxima in \eqref{E:lin-choice} and \eqref{E:hin-choice}
are  achieved at the first term, we get:
\[\begin{split} h_{i,n}(y)-l_{i,n}(x)&=\sy_{D_{k+1}}+\dots+\sy_{D_{\tau_i}}+\sy_{F_{\sigma_i}}+\dots+\sy_{F_{m+1}}\\&+\sum_{w\in F\subsetneq
F_{m+1},~F\in\mh}\sy_F+\sum_{u\in D\subsetneq
D_{k+1},~D\in\mh}\sy_D -d(u,x)-d(w,y)\\&\le
d(u,w)-d(u,x)-d(w,y)\le d(y,x), \end{split}\] where the first
inequality in the last row uses inequality {\bf (D2)} for the edge
joining $u$ and $w$.

If the maxima are equal to $0$, we obtain:
\begin{equation}\label{E:Both0} h_{i,n}(y)-l_{i,n}(x)=\sy_{D_{k+1}}+\dots+\sy_{D_{\tau_i}}+\sy_{F_{\sigma_i}}+\dots+\sy_{F_{m+1}}\le
d(y,x), \end{equation} where the ultimate inequality follows from
the fact that \[d(x,u)-\sum_{u\in D\subsetneq D_{k+1},~
D\in\mh}\sy_D<0\] implies that $x$ is inside $U_D$ for some proper
subset $D\subset D_{k+1}, D\in\mh$. Similarly, \[d(y,w)-\sum_{w\in
F\subsetneq F_{m+1},~ F\in\mh}\sy_F< 0\] implies that $y$ is inside
$U_F$ for some proper subset $F\subset F_{m+1}, F\in\mh$. This
implies the inequality \eqref{E:Both0}. ``Mixed'' cases can be
treated in a ``mixed'' way.

Equation \eqref{E:In_r} allows to evaluate
\[t_{i,n}(y_i)-t_{i,n}(x_i)=\sy_{D_1}+\dots+\sy_{D_{\tau_i}}+\sy_{F_{\sigma_i}}+\dots+\sy_{F_1}.\]
On the other hand, \eqref{E:WeighAtt} implies that this sum is
equal to $d(x_i,y_i)$.
\medskip

For the sequel,  representation of functions $t_{i,n}$ as sums of functions $s_{\lambda,\theta,H}$
are needed.  Our next goal is to prove the following identity:
\begin{equation}\label{E:In_s}\begin{split}
t_{i,n}(x)&= s_{0,1,D_1}(x)+s_{0,1,D_2}(x)+ s_{0,1,D_3}(x)+\dots+
s_{0,1,D_{\tau_i}}(x)\\&+
s_{\sy_{F_{\sigma_i}},-1,F_{\sigma_i}}(x)+\dots +
s_{\sy_{F_2},-1,F_2}(x)+s_{\sy_{F_1},-1,F_1}(x).
\end{split}
\end{equation}

Proof consists in checking that formulas \eqref{E:In_r} and
\eqref{E:In_s} lead to the same values on different pieces of the
space:

\begin{itemize}

\item $d(x,x_i)$ on $B_X(x_i,\sy_{D_1})$

\item $\min\{r_{0,1,D_1}(x), r_{\sy_{D_1},1,D_2}(x)\}$ on
$U_{D_2}$

\item $\min\{r_{0,1,D_1}(x), r_{\sy_{D_1},1,D_2}(x),
r_{\sy_{D_1}+\sy_{D_2},1,D_2}(x)\}$ on $U_{D_3}$,

\item and so on.

\item Also, both formulas lead to the value
$\sy_{D_1}+\dots+\sy_{D_{\tau_i}}$ outside the union of all sets
of the form $U_F$, $F\in\mh$.
\end{itemize}

To finalize  the proof of biorthogonality of $\{t_{i,n},
\fs_i\}_{i=1}^n$ (see the paragraph preceding \eqref{E:P_n}), let
$j\ne i$, $j\in\{1,\dots,n\}$. By  condition \eqref{E:Int1},
$x_j\in D_k$ implies $y_j\in D_k$, the same for $F_k$.
Consequently, either there is a unique $k$ such that $x_j,y_j\in
D_k$, but not in $D_{k-1}$, or there is a similar statement for
$F_k$. Then $U_{\{x_j\}}\ne U_{\{y_j\}}$ and both of them are
contained in $U_{D_k}$. Therefore, by the definition of
$s_{\lambda,\theta,H}$, the equality
$s_{\lambda,\theta,H}(x_j)=s_{\lambda,\theta,H}(y_j)$ holds for
all $\lambda\in\mathbb{R}$, $\theta=\pm 1$, and $H\in\{D_1,\dots,
D_{\tau_i}, F_1,\dots, F_{\sigma_i}\}$. Hence
$t_{i,n}(x_j)-t_{i,n}(y_j)=0$, this completes  our proof of
biorthogonality of   $\{t_{i,n}, \fs_i\}_{i=1}^n$.\medskip

Now we shall make necessary preparations for the proof of
$\|P_n(f)\|_\tc\le \|f\|_\tc$ for every $f\in\tc(X)$ of the form
$f=\1_w-\1_z$ for $w,z\in X$.

We start by picking one of the smallest $D\in\mh$ satisfying $w\in
U_D$ (ties may be resolved arbitrarily) and one of the smallest
$F\in\mh$ satisfying $z\in U_F$. Note that it is possible that $w$
or $z$ are not in $U_D$ for any $D\in\mh$. Without loss of
generality, the possible options for $D$ and $F$ can be listed as:

\begin{enumerate}[{\bf (a)}]

\item\label{I:a} $D\cap F=\emptyset$

\item\label{I:b} $D\subsetneq F$

\item \label{I:c} $D=F$

\item\label{I:d}  $F$ is undefined; meaning that $z$ is not
contained in $U_H$ for any $H\in\mh$.

\item\label{I:e} $D$ and $F$ are undefined,  that is, $w$ and $z$
are not contained in $U_H$ for any $H\in\mh$.

\end{enumerate}

At this stage, it has to be proven that

\begin{equation}\sum_{i=1}^n |t_{i,n}(z)-t_{i,n}(w)|\le d(z,w).
\end{equation}

To do this, in Case \eqref{I:a} we use the following argument. Let
$D_{\beta(z,w)}\supseteq D$ be the largest set in $\mh$ satisfying
$z\notin U_{D_{\beta(z,w)}}$. Likewise, let
$F_{\zeta(w,z)}\supseteq F$ be the largest set satisfying $w\notin
U_{F_{\zeta(w,z)}}$. Next, also, let $D=D_1\subsetneq
D_2\subsetneq\dots\subsetneq D_{\beta(z,w)}$ and $F=F_1\subsetneq
F_2\subsetneq\dots\subsetneq F_m$ be the maximal chains of subsets
in $\mh$.\medskip

The corresponding functions $t_{D,F}$ are constructed in the same
manner as $t_{i,n}$:

\begin{equation}\label{E:t_D,H}\begin{split} t_{D,F}(x)&=
s_{0,1,D_1}(x)+s_{0,1,D_2}(x)+ s_{0,1,D_3}(x)+\dots+
s_{0,1,D_{\beta(z,w)}}(x)\\&+
s_{\sy_{F_{\zeta(w,z)}},-1,F_{\zeta(w,z)}}(x)+\dots +
s_{\sy_{F_2},-1,F_2}(x)+s_{\sy_{F_1},-1,F_1}(x).
\end{split}
\end{equation}

The fact that $t_{D,F}$ is $1$-Lipschitz can be checked in the
same way as for $t_{i,n}$.

Note that each of the summands in the right-hand side of
\eqref{E:t_D,H} appears  in exactly one of the sums in
\eqref{E:In_s}. It cannot be present in several because the vertex
leading to oddness of the cut related  to a set $H\in\mh$ should
be in the corresponding pair $\{x_i,y_i\}$.

Therefore, on one hand,

\begin{equation}\label{E:tDF_as_tin}
t_{D,F}(x)=\sum_{i=1}^nt_{D,F,i}(x),\end{equation} where
$t_{D,F,i}(x)$ is the sum of those summands in \eqref{E:t_D,H}
which are present in the decomposition of $t_{i,n}$.

On the other hand, since $t_{D,F}$ is $1$-Lipschitz, one gets:
\begin{equation}\label{E:t_DH_vs d}
|t_{D,F}(w)-t_{D,F}(z)|\le d(w,z).
\end{equation}

As a result,
\[\begin{split}
t_{D,F}(z)-t_{D,F}(w)&=\sum_{i=1}^n
(t_{D,F,i}(z)-t_{D,F,i}(w))\\&= \sum_{i=1}^n
|t_{D,F,i}(z)-t_{D,F,i}(w)|\\&=\sum_{i=1}^n
|t_{i,n}(z)-t_{i,n}(w)|.\end{split}
\]

The latter equalities have been established with the help of the
following observations: (1) All summands in \eqref{E:t_D,H} have
larger value at $z$ than at $w$; (2) All other functions of the
form $s_{\lambda,\theta,H}$ have the same values at $z$ and $w$.
To see this, it suffices to use the definition of
$s_{\lambda,\theta,H}$ in cases where $U_H$ contains either none
or both of $w,z$, in the latter case, we assume also that
$H\in\mh$ is not the smallest set for which this happens.
\medskip

Case \eqref{I:b}: Consider the maximal increasing sequence of sets
in $\mh$ of the form $D=D_1\subsetneq D_2\subsetneq\dots\subsetneq
D_n=F$. Construct  the function $t_{D,F}$ and complete the proof
as will be described in  Case \eqref{I:d}.\medskip

Case \eqref{I:c}: In this case, define $t_{D,F}$ as $s_{0,1,D}$
and use a simpler version of the argument of Cases \eqref{I:a} and
\eqref{I:b}.
\medskip

In Case \eqref{I:d}, consider the maximal increasing sequence of
sets in $\mh$ of the form $D=D_1\subsetneq
D_2\subsetneq\dots\subsetneq D_{\beta(z,w)}$.

Form the function $t_{D,F}$ as in  the first line of
\eqref{E:t_D,H}, and repeat the same argument as in Case
\eqref{I:a}.

In Case \eqref{I:e}, $t_{i,n}(w)=t_{i,n}(z)$ for every $i$, whence
the conclusion follows. This completes our proof of Lemma
\ref{L:FiniteMatch}.
\end{proof}

Now we shall use this result to prove Theorem \ref{T:Main}.

\begin{proof} Let a metric space $X$ be such that $\tc(X)$
contains a linear isometric copy of $\ell_1$. By Theorem
\ref{T:Contell_1}, this implies that there exists a sequence of
pairs $\{x_i,y_i\}_{i=1}^\infty$ in $X$, with all elements
distinct, such that each set $\{x_iy_i\}_{i=1}^n$ of edges is a
minimum weight perfect matching in the $K(\{x_i,y_i\}_{i=1}^n)$.

For each $n\in\mathbb{N}$, find functionals $\{t_{i,n}\}_{i=1}^n$
by applying Lemma \ref{L:FiniteMatch}.
 The next step in the proof
is to define $1$-Lipschitz functions $\{t_i\}_{i=1}^\infty$ on $X$
as weak$^*$ limits of the sequences $\{t_{i,n}\}_{n=i}^\infty$.
More precisely, we pick a free ultrafilter $\mju$ on $\mathbb{N}$
and let $t_i=w^*-\lim_{i,\mju}t_{i,n}$. (We may understand the
limit as pointwise after replacing  the functions $t_{i,n}$ with
$t_{i,n}(x)-t_{i,n}(O)$ for some base point $O$.) Then, we define
a mapping $P:\tc(X)\to\overline{\lin(\{\fs_i\}_{i=1}^\infty)}$ by
$P(f)=\sum_{i=1}^\infty t_i(f)\fs_i$.

The fact that the sequence $\{\fs_i\}_{i=1}^\infty$ is
isometrically equivalent to the unit vector basis of $\ell_1$ was
observed in \cite{KMO20} (and is easy to check).

Therefore, to justify that the map $P$ is well-defined and at the
same time it is a projection of norm $1$, it suffices to show that
for any $m\in\mathbb{N}$, there holds:
\[\left\|\sum_{i=1}^m t_i(f)\fs_i\right\|_\tc\le\|f\|_\tc.\]

However, this is true because, by Lemma \ref{L:FiniteMatch},
\[\left\|\sum_{i=1}^m t_{i,n}(f)\fs_i\right\|_\tc\le\|f\|_\tc\]
 and, in addition, $\sum_{i=1}^m t_i(f)\fs_i$ is a
(strong) limit of $\sum_{i=1}^m t_{i,n}(f)\fs_i$ as $n\to\infty$
through $\mju$.
\end{proof}

\subsubsection*{Acknowledgement:} The second-named author
gratefully acknowledges the support by the National Science
Foundation grant NSF DMS-1953773.

\renewcommand{\refname}{\section*{References}}

\begin{small}

\end{small}

\textsc{Department of Mathematics, Atilim University, 06830
Incek,\\ Ankara, TURKEY} \par \textit{E-mail address}:
\texttt{sofia.ostrovska@atilim.edu.tr}\par\medskip

\textsc{Department of Mathematics and Computer Science, St. John's
University, 8000 Utopia Parkway, Queens, NY 11439, USA} \par
  \textit{E-mail address}: \texttt{ostrovsm@stjohns.edu}

\end{document}